\documentclass[12pt, reqno,draft]{amsart}

\usepackage{amsmath}
\usepackage{amsthm}
\usepackage{amssymb}
\usepackage[abbrev]{amsrefs}
\usepackage{mathrsfs}
\usepackage[usenames]{color}
\usepackage{yhmath}
\usepackage{bm}
\usepackage{bbm}
\usepackage{enumitem}

\newtheorem{thm}{}[section]
\newtheorem{theorem}[thm]{Theorem}
\newtheorem{corollary}[thm]{Corollary}
\newtheorem{lemma}[thm]{Lemma}
\newtheorem{proposition}[thm]{Proposition}

\theoremstyle{definition}

\theoremstyle{remark}

\newtheorem{remark}[thm]{Remark}

\numberwithin{equation}{section}
\allowdisplaybreaks

\newcommand{\abs}[1]{\left\lvert#1\right\rvert}
\newcommand{\aabs}[1]{\lvert#1\rvert}
\newcommand{\norm}[1]{\left\lVert#1\right\rVert}
\newcommand{\Cu}{\ensuremath{\mathcal{Q}}}

\newcommand{\leb}{\ensuremath{\bm{L}}}
\newcommand{\sldf}{\ensuremath{\bm{\varphi_l^s}}}
\newcommand{\sudf}{\ensuremath{\bm{\varphi_u^s}}}
\newcommand{\ldf}{\ensuremath{\bm{\varphi_l}}}
\newcommand{\udf}{\ensuremath{\bm{\varphi_u}}}
\newcommand{\kl}{\ensuremath{\tilde{\bm{k}}}}
\newcommand{\kk}{\ensuremath{\bm{k}}}
\newcommand{\XB}{\ensuremath{\mathcal{X}}}
\newcommand{\NN}{\ensuremath{\mathbb{N}}}
\newcommand{\uu}{\ensuremath{\bm{u}}}
\newcommand{\xx}{\ensuremath{\bm{x}}}
\newcommand{\UU}{\ensuremath{\mathbb{U}}}
\newcommand{\UB}{\ensuremath{\mathcal{U}}}
\newcommand{\VB}{\ensuremath{\mathcal{V}}}
\newcommand{\GG}{\ensuremath{\mathcal{G}}}
\newcommand{\yy}{\ensuremath{\bm{y}}}
\newcommand{\ww}{\ensuremath{\bm{w}}}
\newcommand{\vv}{\ensuremath{\bm{v}}}
\newcommand{\ee}{\ensuremath{\bm{e}}}
\newcommand{\Ind}{\ensuremath{\mathbbm{1}}}
\newcommand{\EE}{\ensuremath{\mathbb{E}}}
\newcommand{\RR}{\ensuremath{\mathbb{R}}}
\newcommand{\FF}{\ensuremath{\mathbb{F}}}

\newcommand{\Id}{\ensuremath{\mathrm{Id}}}
\newcommand{\XX}{\ensuremath{\mathbb{X}}}
\newcommand{\Fou}{\ensuremath{\mathcal{F}}}
\newcommand{\YY}{\ensuremath{\mathbb{Y}}}
\newcommand{\YB}{\ensuremath{\mathcal{Y}}}
\newcommand{\EB}{\ensuremath{\mathcal{E}}}

\newcommand{\BB}{\ensuremath{\mathcal{B}}}
\newcommand{\Sym}{\ensuremath{\mathbb{S}}}
\newcommand{\zz}{\ensuremath{\bm{z}}}
\newcommand{\hh}{\ensuremath{\bm{h}}}

\DeclareMathOperator{\supp}{supp}

\DeclareMathOperator{\sgn}{sign}
\DeclareMathOperator*{\Ave}{Ave}

\AtBeginDocument{
\def\MR#1{}
}

\begin{document}

\title[]{Sparse approximation using new greedy-like bases in superreflexive spaces}



\thanks{F. Albiac acknowledges the support of the Spanish Ministry for Science and Innovation under Grant PID2019-107701GB-I00 for \emph{Operators, lattices, and structure of Banach spaces}. F. Albiac and J.~L. Ansorena acknowledge the support of the Spanish Ministry for Science, Innovation, and Universities under Grant PGC2018-095366-B-I00 for \emph{An\'alisis Vectorial, Multilineal y Aproximaci\'on.} M. Berasategui was supported by ANPCyT PICT-2018-04104.}

\author[Albiac]{Fernando Albiac}
\address{Department of Mathematics, Statistics, and Computer Sciencies--InaMat2\\
Universidad P\'ublica de Navarra\\
Campus de Arrosad\'{i}a\\
Pamplona\\
31006 Spain}
\email{fernando.albiac@unavarra.es}

\author[Ansorena]{Jos\'e L. Ansorena}
\address{Department of Mathematics and Computer Sciences\\
Universidad de La Rioja\\
Logro\~no\\
26004 Spain}
\email{joseluis.ansorena@unirioja.es}

\author[Berasategui]{Miguel Berasategui}
\address{Miguel Berasategui\\
IMAS - UBA - CONICET - Pab I, Facultad de Ciencias Exactas y Naturales\\
Universidad de Buenos Aires\\
(1428), Buenos Aires, Argentina}
\email{mberasategui@dm.uba.ar}

\begin{abstract}
This paper is devoted to theoretical aspects on optimality of  sparse approximation. We undertake a quantitative study of   new types of greedy-like bases that have recently arisen in the context of nonlinear $m$-term approximation in Banach spaces  as a generalization of the properties that characterize almost greedy bases, i.e., quasi-greediness and democracy.  As a means to compare the efficiency of these new bases with already existing ones in regards to the implementation of the Thresholding Greedy Algorithm, we place  emphasis on obtaining  estimates  for their sequence of unconditionality parameters. Using an enhanced version of the original Dilworth-Kalton-Kutzarova method from \cite{DKK2003} for building almost greedy bases, we manage to construct bidemocratic bases whose unconditionality parameters satisfy significantly worse estimates than almost greedy bases even in Hilbert spaces. 
\end{abstract}

\maketitle

\section{Introduction}\noindent 
The recent developments in the study of the efficiency of the Thresholding Greedy Algorithm (TGA for short) have given rise to new types of greedy like bases which are of interest both from the abstract point of view of functional analysis and also from the more applied nature of the problem of obtaining optimal numerical computations associated to sparse approximation by means of nonlinear algorithms. 

 The TGA simply takes $m$ terms with the maximum absolute values of the coefficients from the expansion of a signal (a function) relative to a fixed representation system (a basis). Different greedy algorithms originate from different ways of choosing the coefficients of the linear combination in the $m$-term approximation to the signal. Another name, commonly used in the literature for $m$-term approximation is \emph{sparse approximation}. Sparse approximation of functions is a powerful analytic tool which is  present in many important applications to image and signal processing, numerical computation, or compressed sensing.
  
 It is fair to say that greedy approximation theory evolved  from the study of the three main types of greedy-like bases, namely, greedy, quasi-greedy, and almost greedy bases. Greedy bases are the best for application of the TGA for sparse approximation, since for any function
  $f$ in a given space $\XX$, after  $m$ iterations it provides approximations with an error of the same order as the best $m$-term theoretical approximation to $f$. On the other hand, that a basis is quasi-greedy merely guarantees that, for any $f\in \XX$, the TGA provides approximants that converge to $f$ but does not guarantee the optimal rate of convergence.

 
 Both greedy and quasi-greedy bases were introduced in the pioneering work of Konyagin and Temlyakov \cite{KoTe1999} from 1999, whereas almost greedy bases were defined shortly afterwards by Dilworth et al.\ \cite{DKKT2003} in what with hindsight would be, together with the work of Wojtaszczyk \cite{Woj2000}, the forerunner article on the functional analytic approach to the theory. If Konyagin and Temlyakov had characterized greedy bases as unconditional bases with the additional property of being democratic, Dilworth et al.\ characterized almost greedy bases as those bases that are simultaneously  quasi-greedy and democratic.
 
 In studying the optimality of the TGA, other  bases have emerged which, despite being more general  than quasi-greedy bases, still preserve essential properties in greedy approximation. Delving deeper into these properties is of interest both from a theoretical and a practical viewpoint. On one hand, isolating the specific features of those bases makes the theory progress; on the other hand, from a more applied approach, working with these properties leads to obtaining sharper estimates for the constants that measure
 the efficiency of the TGA (see \cite{AAB2021}).

In this paper we concentrate on squeeze-symmetric bases and truncation quasi-greedy bases, in a sense that will be made explicit below, with an eye to the quantitative aspects of the theory.
The central question we ask ourselves is whether these bases retain certain relevant numerical features of almost greedy and quasi-greedy bases or not. Answering this question would help us to better acknowledge their role in sparse approximation theory.

From the point of view of sparse approximation in Banach spaces with respect to the TGA,  the most important  numerical information  of a basis $\XB=(\xx_n)_{n=1}^\infty$ is obtained through the sequence $(\leb_m)_{m=1}^\infty$ of its Lebesgue parameters, which, roughly speaking,  measures how far $\XB$ is from being greedy. The growth of these parameters is linearly determined by the combination of the unconditionality parameters $(\kk_m)_{m=1}^\infty$, which quantify the conditionality of $\XB$,  and the squeeze symmetric parameters, which quantify a symmetry property related to democracy (see \cite{AAB2021}*{Theorem 1.5}). In the case when the basis $\XB$ is squeeze-symmetric we have, 
\[
\leb_{m} \approx \kk_{m}, \quad m\in\NN,
\]
hence for this type of bases the growth of the Lebesgue parameters  is completly controlled  by the growth  of the unconditionality parameters. 

The best one can say about the asymptotic estimates for the unconditionality parameters of 
(semi-normalized)
quasi-greedy bases in Banach spaces is that
\begin{equation}\label{eq:GE}
\kk_{m} \lesssim 1+\log m, \quad m\in\NN,
\end{equation}
 (see \cite{DKK2003}*{Lemma 8.2}). 
In turn, an asymptotic upper bound for the unconditionality parameters of truncation quasi-greedy bases in Banach spaces was estimated in \cite{AAW2021b}*{Theorem 5.1}, where it was proved that if $\XB$ is a (semi-normalized) truncation quasi-greedy basis of a Banach space $\XX$ then \eqref{eq:GE} still holds.

Based on this, one might feel tempted to conjecture that, in spite of the fact that truncation quasi-greedy bases are a weaker form of quasi-greediness, the efficiency of the greedy algorithm for the former kind of bases is the same as the efficiency we would get for the latter. There are recent results of a more qualitative nature that substantiate this guess, such as \cite{AABW2021}*{Theorem 9.14}, \cite{AAW2021}*{Theorem 4.3}, \cite{AABW2021}*{Proposition 10.17(iii)}, \cite{AABBL2022}*{Corollary 4.5} and \cite{AABBL2021}*{Corollary 2.6},
all of which are generalizations to truncation quasi-greedy bases of results previously obtained for quasi-greedy bases. These results   improve \cite{AlbiacAnsorena2016}*{Theorem 3.1}, \cite{DSBT2012}*{Theorem 4.2},  \cite{DKKT2003}*{Proposition 4.4}, \cite{DKK2003}*{Corollary 8.6} and \cite{DKKT2003}*{Theorem 5.4}, respectively.

It is therefore crucial to determine whether truncation quasi-greedy  bases provide the same accuracy  in $m$-term greedy approximation  as quasi-greedy bases,  in  general Banach spaces or under certain smoothness conditions of the space.  Imposing superreflexivity to the underlying Banach space is indeed a very natural restriction that leads to an improvement of the performance of the TGA. For instance, it was shown in  \cite{AAGHR2015}*{Theorem 1.1} that the unconditionality parameters of  quasi-greedy bases  in these spaces  satisfy the sharper estimate 
\begin{equation}\label{AAGHRestimate}
\kk_{m}\lesssim (1+\log m)^{1-\epsilon}, \quad m\in\NN,
\end{equation}
for some $0<\epsilon<1$ depending on the basis and the space.

 In this paper  we disprove the guess that the estimate \eqref{AAGHRestimate} should pass to truncation quasi-greedy bases of super-reflexive Banach spaces (see \cite{AABBL2022}*{Remark 3.9})
 by building  squeeze-symmetric bases with ``large'' unconditionality parameters even inside  Hilbert spaces.
To the best of our knowledge, this provides the first evidence of a different behavior between the implementation of TGA for quasi-greedy bases and truncation quasi-greedy bases. In fact, the bases we construct  belong to the more demanding class of bidemocratic bases! Thus, our results connect with and give more relevance to the first known examples of bidemocratic bases which are not quasi-greedy (see \cite{AABBL2021}). 
 
 
 The method we use in our construction is  of interest in the theory by itself since it permits to extend the validity of the Dilworth-Kalton-Kutzarova method (DKK method for short)  to a less restritive class of  bases than the ones considered in \cite{DKK2003}
and \cite{AADK2018}.  The DKK method was invented in \cite{DKK2003} with the  purpose of constructing almost greedy bases in separable Banach spaces which contain a complemented symmetric basic sequence. 
For the original  DKK method to work, the main ingredients are a semi-normalized Schauder basis $\XB$ of a Banach space $\XX$ and a subsymmetric sequence space, from which we obtain a sequence space $\YY$ whose unit vector system is a Schauder basis  fulfilling some special features. Our contribution here consists in  being able to implement the DKK method with  `bases' $\XB$ which, one  one hand, are not necessarily Schauder bases and, on the other hand, need not be semi-normalized. We then study how this extension influences the properties of the resulting basis of the space $\YY$, with the intention to investigate its performance relative to the TGA.


\section{Background and terminology}

\noindent Throughout this paper we will use standard notation and terminology from Banach spaces and greedy approximation theory, as can be found, e.g., in  \cite{AlbiacKalton2016}. We also refer the reader to the recent article \cite{AABW2021} for other more especialized notation. We next single out however the most heavily used terminology. 
 
Let $\XX$ be an infinite-dimensional separable Banach space (or, more generally, a quasi-Banach space) over the real or complex field $\FF$.  We will denote by $\langle B\rangle$ the linear span of a subset $B$ of $\XX$. In turn, $[B]$ denotes the closed linear span of  $B$. Throughout this paper by a \emph{basis} of $\XX$ we mean a  sequence $\XB=(\xx_n)_{n=1}^\infty$ that generates the entire space, in the sense that $[\XB]=\XX$, and for which there is a (unique)  sequence $\XB^*=(\xx_{n}^*)_{n=1}^\infty$ in the dual space $\XX^{\ast}$ such that $\xx_n^*(\xx_k)=\delta_{k,n}$ for all $k$, $n\in\NN$. We will refer to the basic sequence $\XB^*$ in $\XX^*$ as to the \emph{dual basis} of $\XB$.
If the linear span of $\XB^*$ is $w^*$-dense is $\XX^*$, i.e., if  the \emph{coefficient transform}, given by
\begin{equation*}
\Fou=\Fou[\XB,\XX]\colon\XX\to\FF^\NN, \quad f\mapsto (\xx_n^*(f))_{n=1}^\infty,
\end{equation*}
is one-to-one, we say that the basis $\XB$ is \emph{total}.

The \emph{support} of $f\in\XX$ with respect to $\XB$ is the set
\[
\supp(f)=\{ n\in\NN \colon \xx_n^*(f) \not=0\}.
\]

Let $\EB=(\ee_n)_{n=1}^\infty$ be the unit vector system of $\FF^\NN$, and let $\EB^*=(\ee_n^*)_{n=1}^\infty$ be the unit functionals defined for $f=(a_k)_{k=1}^\infty\in\FF^\NN$ and $n\in\NN$  by $\ee_n^*(f)=a_n$. A \emph{sequence space} will be a quasi-Banach space $\YY$ such that $c_{00}\subseteq \YY\subseteq \FF^\NN$ such that $c_{00}$ is a dense subset of $\YY$. Note that, if $\YY$ is a sequence space, $\EB$ is a total basis of $\YY$ whose dual basis is $(\ee_n^*|_\YY)_{n=1}^\infty$. Conversely, given a quasi-Banach space $\YY$ in which  $c_{00}$ is a dense subset, the sequence $\EB$ is a basis if and only if $\ee_n^*|_{c_{00}}$ is bounded for all $n\in\NN$; and $\EB$ is a total basis if and only if there is a one-to-one continuous extension $T\colon\YY\to\FF^\NN$ of the identity map on $c_{00}$. Here, we consider $\FF^\NN$ endowed with the pointwise convergence topology. If $\YY$ is a sequence space we will identify its dual space $\YY^*$ with the sequence space consisting of all $g\in\FF^\NN$ such  $ \langle \cdot , g \rangle $ restricts to a functional of $\YY$, where   $ \langle \cdot , \cdot \rangle $ is the canonical dual pairing defined for $f=(a_n)_{n=1}^\infty\in c_{00}$ and $g\in(b_n)_{n=1}^\infty\in\FF^\NN$ by
\[
\left\langle f, g\right\rangle=\sum_{n=1}^\infty a_n b_n.
\]

Let $\EE$ denote the set  of all scalars of modulus one. Given $A\subseteq\NN$ finite and $\varepsilon=(\varepsilon_n)_{n\in A}\in\EE^A$ we put
\[
\Ind_{\varepsilon,A}[\XB,\XX]=\sum_{n\in A} \varepsilon_n\, \xx_n.
\] 
If $\varepsilon_n=1$ for all $n\in A$, we set $\Ind_{A}[\XB,\XX]=\Ind_{\varepsilon,A}[\XB,\XX]$. 
 
 A basis $\XB$ of a quasi-Banach space $\XX$  is said to be \emph{democratic} (resp., \emph{super-democratic}) if there is a constant $C\ge 1$ such that 
 \begin{equation*}
 \norm{ \Ind_{A}[\XB,\XX]} \le C \norm{ \Ind_{B}[\XB,\XX]}  \; (\mbox{resp., } \; \norm{ \Ind_{\varepsilon,A}[\XB,\XX]} \le C \norm{ \Ind_{\delta,B}[\XB,\XX]} )
 \end{equation*}
 for all finite subsets $A$ and $B$ of $\NN$ with $\abs{A} \le \abs{B}$, all $\varepsilon\in\EE^A$, and all $\delta\in\EE^B$. If the above inequality holds for a given $C$, we say that $\XB$ is $C$-democratic (resp., $C$-super-democratic).

To measure the democracy of a basis $\XB$ of a quasi-Banach space $\XX$, we introduce the \emph{upper and lower democracy functions} of the basis, defined for $m\in\NN$ by
\begin{align*}
\udf[\XB,\XX](m)&=\sup\left\{ \norm{\Ind_{A}[\XB,\XX]} \colon \abs{A} \le m\right\} \mbox{ and}\\
\ldf[\XB,\XX](m)&=\inf\left\{ \norm{\Ind_{A}[\XB,\XX]} \colon \abs{A} \ge m\right\},
\end{align*}
respectively. The basis $\XB$ is democratic if and only if 
\[
\udf[\XB,\XX]\le C \ldf[\XB,\XX].
\] 
Similarly, the basis is $C$-super-democratic if and only if 
\[
\sudf[\XB,\XX]\le C\sldf[\XB,\XX],
\]
where  $\sudf[\XB,\XX]$ and $\sldf[\XB,\XX]$ are, respectively, the \emph{upper and lower super-democracy} defined for $m\in\NN$ as
\begin{align*}
\sudf[\XB,\XX](m)&=\sup\left\{ \norm{\Ind_{\varepsilon,A}[\XB,\XX]} \colon \abs{A} \le m, \, \varepsilon\in\EE^A\right\},\\
\sldf[\XB,\XX](m)&=\inf\left\{ \norm{\Ind_{\varepsilon,A}[\XB,\XX]} \colon \abs{A} \ge m, \, \varepsilon\in\EE^A\right\}.
\end{align*}
The upper super-democracy function of a basis,  also called the \emph{fundamental function} of the basis, grows as the upper democracy function. In contrast, the lower super-democracy function of the basis can grow much more slowly than the lower democracy function (see \cite{Woj2014}).


A basis $\XB=(\xx_n)_{n=1}^\infty$ is said to be \emph{symmetric} (resp., \emph{subsymmetric}) if there is a constant $C$ such that 
\[
\frac{1}{C} \norm{\sum_{n=1}^\infty a_n \, \xx_n}\le 
 \norm{\sum_{n=1}^\infty \varepsilon_n\, a_n \, \xx_{\pi(n)}} \le C \norm{\sum_{n=1}^\infty a_n \, \xx_n}
\]
for all $(a_n)_{n=1}^\infty\in c_{00}$, all $(\varepsilon_n)_{n=1}^\infty\in\EE^\NN$, and all bijective (resp., increasing) maps  $\pi\colon\NN\to\NN$. 
If we can choose $C=1$, we say that $\XB$ is $1$-symmetric (resp., $1$-subsymmetric). Any symmetric  (resp., subsymmetric) basis is $1$-symmetric (resp., $1$-subsymmetric) under a suitable renorming of the space (see \cites{Singer1970, Ansorena2018}). Moreover,  $1$-symmetric bases are $1$-subsymmetric. Here, we will deal with \emph{symmetric and subsymmetric sequence spaces}, i.e., sequence spaces whose unit vector system is a $1$-symmetric or $1$-subsymmetric basis. The unit vector system of a subsymmetric sequence space, besides $1$-unconditional, is $1$-super-democratic, i.e., for each $m\in\NN$ there is a constant  $\Lambda_{m}=\Lambda_m[\Sym]\in(0,\infty)$ such that
\begin{center}
$\Lambda_m[\Sym]=\norm{\Ind_{\varepsilon,A}[\EB,\Sym]}$  $\forall\varepsilon\in\EE^A$ and $\forall A\subseteq \NN$ with $\abs{A}=m$.
\end{center}
We call $(\Lambda_m[\Sym])_{m=1}^\infty$ the fundamental function of $\Sym$. The sequence $(\Lambda_m[\Sym])_{m=1}^\infty$ is non-decreasing, and, in case that $\Sym$ is a Banach space, so is $(m/\Lambda_m[\Sym])_{m=1}^\infty$ \cite{DKKT2003}. In fact, the closed linear span $\Sym_0^*$ of the unit vector system of $\Sym^*$ is a subsymmetric sequence space with
\begin{equation}\label{eq:SubsBidem}
\Lambda_m[\Sym_0^*] \approx \frac{m}{\Lambda_m[\Sym]}, \quad m\in\NN
\end{equation}
(see \cite{LinTza1977}). 

In the following sections we will use subsymmetric sequence spaces whose fundamental function grows in a controlled manner, and the geometry of the underlying  space plays an important role in order to ensure this steady behaviour. The next two regularity conditions formalize that pattern. We say that a sequence $(\Gamma_m)_{m=1}^\infty$ in $(0,\infty)$  has the \emph{lower regularity property} (LRP for short) if there a positive integer $b$ such 
 \[
 2\Gamma_m \le \Gamma_{bm}, 
\quad
 m\in\NN.
 \]
We say that  $(\Gamma_m)_{m=1}^\infty$  it has the \emph{upper regularity property} (URP for short) if there a positive integer $b$ such 
 \[
 2\Gamma_{bm} \le b \Gamma_{m}, 
\quad
 m\in\NN.
 \]
\begin{proposition}[\cite{DKKT2003}*{Proposition 4.1}]\label{prop:RadLURP}
Let $\Sym$ be a subsymmetric sequence space.
\begin{enumerate}[label={(\roman*)},leftmargin=*,widest=ii]
\item If $\Sym$ has some nontrivial cotype, then $(\Lambda_m[\Sym])_{m=1}^\infty$ has the LRP. 
\item  If $\Sym$ has some nontrivial type, then $(\Lambda_m[\Sym])_{m=1}^\infty$ has the LRP and the URP.
\end{enumerate} 
\end{proposition}
From another point of view, the lattice structure induced on $\Sym$ by its unit vector system yields that $\Sym$ has some nontrivial cotype if and only if it has some nontrivial concavity, and it has some nontrivial type if and only if it is superreflexive (see \cite{LinTza1979}).  

For further reference, we record a regularity result 
that we will need.
 \begin{lemma}[See \cite{AlbiacAnsorena2016}]\label{lem:ULRP}
Let $(\Gamma_m)_{m=1}^\infty$ be a sequence in $(0,\infty)$ such that $(m/\Gamma_m)_{m=1}^\infty$  is non-decreasing. Then $(\Gamma_m)_{m=1}^\infty$ has the LRP if and only if $(m/\Gamma_m)_{m=1}^\infty$ has the URP. Moreover, if $(\Gamma_m)_{m=1}^\infty$ has the LRP then it satisfies the Dini condition
 \[
 \sum_{n=1}^m \frac{\Gamma_n}{n} \approx \Gamma_m, \quad m\in\NN.
 \]
\end{lemma}

We say that a basis $\XB$ of a quasi-Banach space $\XX$ is \emph{squeeze-symmetric} if there are symmetric sequence spaces $\Sym_1$ and $\Sym_2$ such that 
\begin{itemize}[leftmargin=*]
\item the \emph{series transform}, defined by
$(a_n)_{n=1}^\infty\mapsto \sum_{n=1}^\infty a_n \, \xx_n$ is a bounded operator from $\Sym_1$ into $\XX$,
\item the coefficient transform  is a bounded operator from $\XX$ into $\Sym_2$, and 
\item the spaces $\Sym_1$ and $\Sym_2$ are close to each other in the sense that
\[
\Lambda_m[\Sym_1] \approx \Lambda_m[\Sym_2], \quad m\in\NN.
\]
\end{itemize}

If a basis is squeeze-symmetric  then it is democratic, and  its fundamental function is equivalent to the fundamental function of the  symmetric sequence spaces that sandwhich it. It is known that these symmetric sequence spaces can be chosen to be sequence Lorentz sequence spaces. Let us briefly recall their definition. 

A \emph{weight} will be a non-negative sequence $(w_n)_{n=1}^\infty$ with $w_1>0$. Given $0< q<\infty$ and a weight $\ww=(w_n)_{n=1}^\infty$, the  Lorentz sequence space $d_q(\ww)$ consists of all sequences $f$ in $c_{0}$ whose non-increasing rearrangement $(a_n^*)_{n=1}^\infty$ satisfies
 \[
\norm {f }_{d_q(\ww)}= \left( \sum_{n=1}^\infty ( s_n a_n^*)^q \frac{w_n}{s_n} \right)^{1/q} <\infty,
\]
where $s_n=\sum_{k=1}^n w_k$. In turn, the weak Lorentz sequence space $d_\infty(\ww)$ consists of all sequences $f=(a_n)_{n=1}^\infty\in c_0$ whose non-increasing rearrangement $(a_n^*)_{n=1}^\infty$ satisfies
\[
\norm{ f}_{d_\infty(\ww)}=\sup_m a_n^* s_n<\infty.
\]
We have $\Lambda_m[d_q(\ww)] \approx s_m$ for $m\in\NN$. Moreover if  $0< p\le q\le \infty$,
\begin{equation*}
\norm {f }_{d_q(\ww)} \lesssim \norm {f }_{d_p(\ww)}, \quad f\in c_0.
\end{equation*}
Although Lorentz sequence spaces are named after the weight $\ww$, they rather depend on the primitive sequence $(s_m)_{m=1}^\infty$. In fact, we have the following result.
\begin{lemma}[see \cite{AABW2021}*{\S9}]\label{lem:LorentzEq}
Let $\ww(w_n)_{n=1}^\infty$ and $\ww'=(w_n')_{n=1}^\infty$ be weights, and let $0<q\le\infty$. Then $\norm {f}_{d_q(\ww)} \approx \norm {f }_{d_p(\ww')}$ for $f\in c_0$ if and only if $\sum_{n=1}^m w_n \approx \sum_{n=1}^m w_n'$ for $m\in\NN$.
\end{lemma}
We refer the reader to \cite{AABW2021}*{\S9} for background on this kind of spaces.

A basis $\XB$ of a quasi-Banach space $\XX$ is squeeze-symmetric if an only if there are a weight $\ww$ and $0<q<\infty$ such that the  series transform is a bounded operator from $d_q(\ww)$ into $\XX$, and the coefficient transform is a bounded operator from $\XX$ into $d_\infty(\ww)$. Moreover, if $\XX$ is $p$-convex for some $0<p\le 1$, we can pick $p=q$.

A basis $\XB$ is said to be \emph{bidemocratic} if and only if 
\[
\sup_m \frac{1}{m} \udf[\XB,\XX](m) \udf[\XB^*,\XX^*](m) <\infty.
\]
It is known that  bidemocratic bases are in particular squeeze-symmetric \cite{AABW2021}.




In consistency with the characterizations of   greedy bases and almost greedy bases, squeeze-symmetric bases can be characterized as well as those democratic bases that satisfy an additional unconditionality-like condition.  To describe this condition, which defines truncation quasi-greedy bases and  plays a key role in the characterization of almost greedy bases as those bases that are simultaneously quasi-greedy and democractic, we need to introduce some additional terminology. 

Usually, the TGA is studied for bases $\XB=(\xx_n)_{n=1}^\infty$ that are \emph{semi-normalized}, i.e.,
\[
0<\inf_n\norm{ \xx_n} \le \sup_n \norm{ \xx_n}<\infty,
\]
and $M$-bounded, i.e.,
\[
\sup_n \norm{ \xx_n} \norm{ \xx_n^*} <\infty.
\]
That is, it is usual to assume that both $\XB$ and $\XB^*$ are norm-bounded. For the purposes of this paper, however, it will be convenient not to take for granted these assumptions a priori. 

Since the  coefficient transform maps $\XX$ into $c_0$ if and only if $\XB^*$ is norm-bounded,  there could be vectors $f\in\XX$ for which the TGA $(\GG_m(f))_{m=1}^\infty$ is not defined. To circumvent this initial drawback, we will consider greedy-type properties of $\XB$ in terms of greedy sets and greedy projections.  

A finite subset $A\subseteq\NN$ is a \emph{greedy set} of $f\in\XX$ with respect to the basis $\XB$ if $\abs{\xx_{n}^{\ast}(f)}\ge \abs{\xx_{k}^{\ast}(f)}$ whenever $n\in A$ and $k\not\in A$. 

Let $\sgn(\cdot)$ be the sign function, defined for $\lambda\in\FF\setminus\{0\}$ as $\sgn(\lambda)=\lambda/\abs{\lambda}$, and  $\sgn(0)=1$. Given a basis $\XB$ of a quasi-Banach space $\XX$ we put
\[
\varepsilon(f)=(\sgn(\xx_n^*(f)))_{n=1}^\infty\in\EE^\NN.
\]
A basis $\XB$ is said to be \emph{truncation-quasi-greedy} if there is a constant $C$ such that 
\begin{equation*}
\min_{n\in A} \abs{\xx_n^*(f)} \norm{ \Ind_{\varepsilon(f),A}[\XB,\XX]} \le C \norm{f}
\end{equation*}
for all $f\in\XX$ and all greedy sets $A$ of $f$. If the above holds for a given constant $C$, we say that $\XB$ is truncation-quasi-greedy with constant $C$. It is known \cite{AABW2021} that  quasi-greedy bases are truncation-quasi-greedy. In turn,  truncation-quasi-greedy bases are \emph{unconditional for constant coefficients} (UCC for short), i.e., 
\begin{equation}
\norm{ \Ind_{\varepsilon,A}[\XB,\XX]}\le C \norm{ \Ind_{\varepsilon,B}[\XB,\XX]}
\label{SUCC}
\end{equation}
for all $B\subseteq\NN$ finite, all $A\subseteq B$, and some constant $C$. We also point out that a basis is super-democratic if and only if it is democratic and UCC. 
Given  a basis $\XB$ and a finite subset $A$ of $\NN$, the \emph{coordinate projection}  onto the subspace $[\xx_{n}\colon n\in A]$ is the linear operator $S_A[\XB,\XX]\colon \XX\to \XX$ given by
 \[
 f\mapsto \sum_{n\in A} \xx_n^*(f) \, \xx_n.
\]
Since the basis $\XB$ is \emph{unconditional} if and only if the operators $S_A[\XB,\XX]$ are uniformly bounded, to quantify how far  a basis is from being uncoditional it is customary to use the  \emph{unconditionality parameters}
\[
\kk_m[\XB,\XX]=\sup\left\{  \norm{ S_A[\XB,\XX]}  \colon  A\subseteq\NN, \ \abs{A} \le m \right\}, \quad m\in\NN.
\]

Notice that a basis is $M$-bounded if and only if
$
\sup_m \norm{S_{\{m\}}}<\infty.
$
Hence, an $M$-bounded basis $\XB$ of a Banach space $\XX$  satisfies the
estimate
\begin{equation}\label{eq:BanachEstimate}
\kk_m[\XB,\XX] \lesssim m, \quad m\in\NN.
\end{equation}


We will use other unconditionality-type parameters as  instruments to obtain information on the growth of $(\kk_m)_{m=1}^\infty$. To that end, for $m\in \NN$ we put
\[
\kl_m[\XB,\XX]=\sup\left\{  \norm{ S_A[\XB,\XX](f)} \colon f \in B_\XX \cap [\xx_j\colon 1\le j \le m], \, A\subseteq\NN\right\},
\]
where $B_\XX$ denotes the closed unit ball of $\XX$. Note that $\kl_m\le \kk_m$ for all $m\in\NN$. 

Given $m\in\NN\cup\{0\}$, we set $S_m=S_{\{1,\dots,m\}}$. Note that $S_{\{m\}}=S_m-S_{m-1}$ for all $m\in\NN$, and that $\XB$ is a Schauder basis if and only if $\sup_m \norm{S_m}<\infty$. Thus, any Schauder basis is $M$-bounded. Let us emphasize here that the celebrated theorem of Enflo \cite{Enflo1973} that proves the existence of a separable Banach space without a Schauder basis does not hold for $M$-bounded bases.

\begin{theorem}[See \cite{HMVZ2008}]\label{thm:ExistsMB}
Every separable Banach space has an $M$-bounded basis.
\end{theorem}

It is clear that democratic bases are semi-normalized. In turn, as we next show, truncation-quasi-greedy bases are $M$-bounded. Recall that a \emph{block basic sequence} of a basis $\XB=(\xx_n)_{n=1}^\infty$ of a quasi-Banach space $\XX$ is a sequence $\YB=(\yy_j)_{j=1}^\infty$ in $\XX$ of the form
\[
\yy_j=\sum_{n\in D_j} a_n\, \xx_n,
\]
for some sequence $(D_j)_{j=1}^\infty$ of pairwise disjoint nonempty finite subsets of $\NN$ and some sequence $(a_n)_{n=1}^\infty$ in $\FF$ with $a_n\not=0$ for all $n\in D_j$ and all $j\in\NN$. If, for each $j\in\NN$, $\abs{a_n}$ is constant on $D_j$,  we say that $\YB$ is a \emph{constant-coefficient block basic sequence} of $\XB$.

\begin{lemma}\label{lem:TQGvsMB}
Any constant-coefficient block basic sequence of a truncation quasi-greedy basis is an $M$-bounded basis of its closed linear span.
\end{lemma}

\begin{proof}
Let $\XB$ be a truncation quasi-greedy basis  of a quasi-Banach space $\XX$  with constant $C\ge 1$. Let $(D_j)_{j=1}^\infty$ be a pairwise disjoint  sequence  of finite subsets of $\NN$, and let $\varepsilon\in \EE^{\NN}$. We need to prove that the sequence
\[
\yy_j:=\Ind_{\varepsilon,D_j}[\XB,\XX], \quad j\in\NN,
\]
is an $M$-bounded basis of $[\yy_j\colon j\in\NN]$. Fix $(a_j)_{j=1}^\infty\in c_{00}$. Given $k\in J$, there are  greedy sets $A$ and $B$ of $
f:=\sum_{j=1}^\infty a_j \, \yy_j
$
such that $A\subseteq B$ and $B\setminus A=D_k$. We have 
\[
\abs{a_k}=\min_{n\in B} \abs{\xx_n^*(f)}\le \min_{n\in A} \abs{\xx_n^*(f)}.
\]
Hence, if $\kappa$ denotes the modulus of concavity of $\XX$,
\begin{align*}
\abs{a_k} \norm{\yy_k} 
&=\abs{a_k} \norm{\Ind_{\varepsilon,B}[\XB,\XX] - \Ind_{\varepsilon,A}[\XB,\XX]} \\
&\le \kappa  \abs{a_k} \norm{\Ind_{\varepsilon,B}[\XB,\XX] } +\kappa \abs{a_k}  \norm{\Ind_{\varepsilon,A}[\XB,\XX]}\\
&\le \kappa C \norm{f}.\qedhere
\end{align*}
\end{proof}

A sequence $(\yy_n)_{n=1}^\infty$ of a quasi-Banach space $\YY$ is said to be an \emph{affinity} of a sequence $(\xx_n)_{n=1}^\infty$ if there is a sequence $(\lambda_n)_{n=1}^\infty$ in $\FF\setminus\{0\}$ such that $\yy_n=\lambda_n \, \xx_n$ for all $n\in\NN$. Suppose that $\YB$ is an affinity of $\XB$. Then, if $\XB$ is a basis, so is $\YB$. Morever, if $\XB$ is $M$-bounded, so is $\YB$, and if $\XB$ is a Schauder basis so is $\YB$. We also note that $\kl_m[\XB,\XX]=\kl_m[\YB,\XX]$ and $\kk_m[\XB,\XX]=\kk_m[\YB,\XX]$ for all $m\in\NN$.

Given quasi-Banach spaces $\XX$ and $\YY$, $\XX\oplus\YY$ stands for its directed sum endowed with the quasi-norm
\[
\norm{ (f,g)}=\max\{ \norm{f}, \norm{g}\}, \quad f\in\XX,\; g\in\YY.
\]

O course, $\XX\oplus\YY$ is a quasi-Banach space. Given a direct sum $\XX\oplus\YY$ we denote by $\pi_1$ and $\pi_2$ the projections onto the first and the second components, respectively.

 If $\XB=(\xx_n)_{n=1}^\infty$ and $\YB=(\yy_n)_{n=1}^\infty$ are sequences in $\XX$ and $\YY$, respectively,  its direct sum is the sequence $\XB\oplus\YB=(\zz_n)_{n=1}^\infty$ in $\XX\oplus\YY$ defined by
\[
\zz_{2n-1}=(\xx_n,0), \quad \zz_{2n}=(0,\yy_n), \quad n\in\NN. 
\]
It is clear that if $\XB$ and $\YB$ are bases, then $\XB\oplus\YB$ is a basis with coordinate functionals $\XB^*\oplus\YB^*$. We set $\XX^2=\XX\oplus\XX$ and $\XB^2=\XB\oplus\XB$.

\section{Extension of the DKK method to  general bases}\label{sect:DKK}

\noindent Let $\Sym$ be a locally convex subsymmetric sequence space. Set 
$
\Lambda_m=\Lambda_m[\Sym]
$
for $m\in\NN$. Let $\sigma=(\sigma_n)_{n=1}^\infty$ be an ordered partition of $\NN$.
Consider the sequence $\VB=\VB[\Sym,\sigma]=(\vv_n)_{n=1}^\infty$ in $\Sym$ given by
\[
\vv_n=\frac{1}{\Lambda_{\aabs{\sigma_n}}}\Ind_{\sigma_n}, \quad n\in\NN.
\]
The sequence $\VB^*=\VB^*[\Sym,\sigma]=(\vv_n^*)_{n=1}^\infty$ in $\Sym^*$ given by
\[
\vv_n^*=\frac{\Lambda_{\aabs{\sigma_n}}}{\abs{\sigma_n}} \Ind_{\sigma_n}^*, \quad n\in\NN
\]
is biorthogonal to $\VB$. By construction, $\VB$ is normalized. In turn, $\VB^*$ is semi-normalized by \eqref{eq:SubsBidem}.

Let $\Ave(f,A)$ denote the average of $f=(a_n)_{n=1}^\infty$ on a finite set $A\subseteq\NN$, i.e.,
\[
\Ave(f,A)=\frac{1}{\abs{A}} \sum_{k\in A} a_k.
\]
Consider the averaging projection $P_\sigma\colon\FF^\NN \to\FF^\NN$ defined by
\[
P_\sigma(f)=(b_k)_{k=1}^\infty, \quad b_k=\Ave(f,\sigma_n) \text{ if } k\in\sigma_n.
\]
Let $Q_\sigma$ be the complementary projection, that is, $Q_\sigma=\Id_{\FF^\NN}-P_\sigma$. By \cite{LinTza1977} or \cite{AADK2018} we have $\norm{ P_\sigma}_{\Sym\to\Sym}\le 2$. Consequently, $\norm{ Q_\sigma}_{\Sym\to\Sym}\le 3$. Note that
\[
P_\sigma(f)=\sum_{n=1}^\infty \vv_n^*(f) \, \vv_n, \quad f\in\FF^\NN.
\]

Let $\XB=(\xx_n)_{n=1}^\infty$ be a linearly independent sequence in a Banach $\XX$ that generates the entire space $\XX$. We define
\begin{align*}
\norm{ f }_{\XB,\Sym,\sigma}&=\norm{ Q_\sigma(f)}_\Sym+\norm{ L[\VB[\Sym,\sigma],\XB](P_\sigma(f))},\\
&=\norm{ Q_\sigma(f)}_\Sym+\norm{ \sum_{n=1}^\infty \vv^*_n(f)\, \xx_n}, \quad f\in c_{00},
\end{align*}
where 
\[L[\YB,\XB]\colon \langle \YB\rangle\to \langle \XB\rangle,\qquad  \yy_n\mapsto \xx_n,
\] stands for the  operator from the linear span of a basis $\YB=(\yy_n)_{n=1}^\infty$ onto the linear span  of a basis $\XB=(\xx_n)_{n=1}^\infty$.
Note that $P_\sigma(f)\in\langle \VB[\Sym,\sigma]\rangle$ for all $f\in c_{00}$. Thus, $\norm{ \cdot}_{\XB,\Sym,\sigma}$ is well-defined.
\begin{lemma}
Let $\Sym$ be a symmetric sequence space, $\sigma$ be an ordered partition of $\NN$ and $\XB$ be a linearly independent sequence in a quasi-Banach space $\XX$. Then $\norm{ \cdot}_{\XB,\Sym,\sigma}$ is a quasi-norm on $\XX$.
\end{lemma}
\begin{proof}
It is clear from definition that $\norm{ \cdot}_{\XB,\Sym,\sigma}$ is a semi-quasi-norm. Assume that $\norm{ f }_{\XB,\Sym,\sigma}=0$. Then $Q_\sigma(f)=0$ and $\vv_n^*(f)=0$ for all $n\in\NN$. Then $P_\sigma(f)=0$ and so $f=P_\sigma(f)+Q_\sigma(f)=0$.
\end{proof}

We denote by $\YY=\YY[\XB,\Sym,\sigma]$ the completion of the  quasi-normed space $(c_{00}, \norm{ \cdot }_{\XB,\Sym,\sigma})$.

\begin{lemma}\label{lem:isomorphism}
Let $\Sym$ be a locally convex subsymmetric sequence space, $\sigma$ be an ordered partition of $\NN$, and $\XB$ be a linearly independent sequence in a quasi-Banach space $\XX$. Suppose that $\XB$ generates $\XX$. Then $ \YY[\XB,\Sym,\sigma]\simeq Q_\sigma(\Sym)\oplus \XX$. To be precise $Q_\sigma(\Sym)\cap c_{00}$ is dense in $Q_\sigma(\Sym)$, and the maps
\begin{align*}
S&:=(Q_\sigma, L[\VB[\Sym,\sigma],\XB]\circ P_\sigma)\colon c_{00}\to (Q_\sigma(\Sym)\cap c_{00}) \oplus \langle \XB\rangle\\
T&:=\pi_1+ L[\XB,\VB[\Sym,\sigma]]\circ \pi_2\colon (Q_\sigma(\Sym)\cap c_{00}) \oplus \langle\XB\rangle \to c_{00}
\end{align*}
are inverse linear bijections that extend to inverse isometries.
\end{lemma}

\begin{proof}
The proof of the corresponding result from \cite{AADK2018} holds in this general setting. Note that
\begin{align*}
S(f)&=\left( Q_\sigma(f), \sum_{n=1}^\infty \vv_n^*(f)\, \xx_n\right),\quad f\in \YY[\XB,\Sym,\sigma]\cap c_{00},\\
T(g,x)&=g+ \sum_{n=1}^\infty \xx_n^*(x)\, \vv_n, \quad g\in Q_\sigma(\Sym)\cap c_{00},\; x\in \langle \XB\rangle.\qedhere
\end{align*}
\end{proof}

\begin{proposition}\label{prop:basisY}
Let $\Sym$ be a symmetric space, let $\sigma$ be an ordered partition of $\NN$ with $\abs{\sigma_n}\ge 2$ for all $n\in\NN$, and let $\XB$ be a linearly independent sequence of a quasi-Banach space $\XX$ with $[\XB]=\XX$. Then the unit vector system $\EB$ is a basis of $\YY=\YY[\XB,\Sym,\sigma]$ if and only if $\XB$ is a basis of $\XX$. Moreover, in the case when $\XB$ is a basis of $\XX$  the following statements hold:
\begin{enumerate}[label={(\roman*)},leftmargin=*,widest=ii]
\item\label{prop:basisY:1} $\EB$ is a total basis of $\YY$  if and only if $\XB$ is a  total basis of $\XX$.

\item\label{prop:basisY:5}
The dual basis of the unit vector system of $\YY$ is equivalent to the unit vector system of $\YY[\BB,\Sym^*,\sigma]$, where $\BB=( \xx_n^*/\norm{\vv_n^*})_{n=1}^\infty$.

\item\label{prop:basisY:2} $\norm{ \ee_n}_{\YY} \approx \max\{ 1, \norm{ \xx_n} \Lambda_{\aabs{\sigma_n}}/ \abs{\sigma_n}\}$ for $n\in\NN$;
\item\label{prop:basisY:3} $ \norm{ \ee_n}_{\YY^*} \approx \max\{ 1, \norm{ \xx_n^*}/ \Lambda_{\aabs{\sigma_n}}\}$ for $n\in\NN$;
\item\label{prop:basisY:4} $\EB$ is a semi-normalized and $M$-bounded basis of $\YY$ if and only if
\[
\norm{ \xx_n}\lesssim \frac{\abs{\sigma_n}}{\Lambda_{\aabs{\sigma_n}}} \;\text{ and }\;
\norm{ \xx_n^*} \lesssim \Lambda_{\aabs{\sigma_n}},\; n\in\NN.
\]

\item\label{prop:basisY:8}  $\kl_{M_n}[\EB,\YY]\gtrsim \kl_n[\XB,\XX]$, where $M_n=\sum_{k=1}^n \abs{\sigma_k}$.

\item\label{prop:basisY:6} Set $\YY_n=[\ee_k \colon k\in\sigma_n]$. Then, $\XB$ is a Schauder basis of $\XX$ if and only if $(\YY_n)_{n=1}^\infty$ is a Schauder decomposition of $\YY$.

\item\label{prop:basisY:9} If $\EB$ is an UCC basis of $\YY$, then $\XB$ is semi-normalized.

\item\label{prop:basisY:12} The block basic sequence $(\Ind_{\sigma_n}[\EB,\YY])_{n=1}^\infty$ is isometrically equivalent to an affinity of  $\XB$.

\item\label{prop:basisY:10} 
Set $\ww=(\Lambda_n - \Lambda_{n-1})_{n=1}^\infty$  and $\uu=(\Lambda_n / n)_{n=1}^\infty$.
Suppose that  $\XX$ is locally convex, that $\XB$ is semi-normalized and $M$-bounded, and that $1+M_{n-1}\lesssim \abs{\sigma_n}$ for $n\in\NN$. Then,
\[
d_1(\uu)\subseteq \YY\subseteq d_\infty(\ww)
\]
(with continuous embeddings).

\item\label{prop:basisY:11} Suppose that $\XX$ is locally convex, $(\Lambda_n)_{n=1}^\infty$ has the LRP, and $M_n\lesssim \abs{\sigma_n}$ for $n\in\NN$. The following are equivalent:
\begin{enumerate}[label={(\arabic*)},leftmargin=*,widest=3]
\item\label{SS:a}  The unit vector system is a squeeze symmetric basis of $\YY$.
\item\label{SS:b}  The unit vector system is a truncation quasi-greedy basis of $\YY$.
\item\label{SS:d} $\XB$ is semi-normalized basis  $M$-bounded basis of $\XX$.
\end{enumerate}

\end{enumerate}
\end{proposition}

\begin{proof}
Given $x\in \langle\XB\rangle$ and $n\in\NN$, let $\xx_n^{\#}(x)$ be the $n$th coordinate of the expansion of $x$ with respect to $\XB$. 
Via the isometry provided by Lemma~\ref{lem:isomorphism}, the functional $\ee_k^*$ corresponds with the map
\begin{equation}\label{eq:zk}
(g,x)\mapsto \zz_k^*(g,x):=\ee_k^*(g) + \frac{1}{\Lambda_{\aabs{\sigma_n}}}\xx_n^{\#}(x),
\end{equation}
where $k\in\sigma_n$. Since $\abs{\ee_k^*(f)}\le \norm{ f}_{\Sym}$ for all $f\in Q_\sigma(\Sym)$, $\ee_k^*$ defines a bounded operator on $\YY$, if and only if $\xx_n^{\#}$ extends to a bounded operator on $\XX$, in which case $\ee_k^*|_\YY$ corresponds with the map 
\[
\zz_k^*\colon Q_\sigma(\Sym) \oplus \XX \to \FF, \quad (g,x) =\ee_k^*(g) + \frac{1}{\Lambda_{\aabs{\sigma_n}}}\xx_n^{*}(x),
\quad k\in\sigma_n.
\]

This expression for $\zz_k^*$ gives that the unit vector system is a total basis if and only if the map
\[
(g,x) \mapsto F(g,x)
=\left(\ee_k^*(g) + \frac{1}{\Lambda_{\aabs{\sigma_n}}}\xx_n^{*}(x)\right)_{n=1}^{\infty}, \quad 
 g\in Q_\sigma(\Sym), \quad x\in \XX,
 \]
is one-to-one.
Notice that, for all  $Q_\sigma(\Sym)$ and $ x\in \XX$,
\[
F(g,x)=g+\sum_{n=1}^\infty  \frac{1}{\Lambda_{\aabs{\sigma_n}}}\xx_n^{*}(x) \Ind_{\sigma_n}
\]
Choosing $g=0$  we obtain the ``only'' if part of \ref{prop:basisY:1}. Assume that $\XB$ is total and that $F(g,x)=0$. Then, averaging on $\sigma_n$ we get $\xx_n^{*}(x)=0$. Consequently, $x=0$. Hence, $g=F(g,x)=0$.

To prove \ref{prop:basisY:5} we set 
\[
\UB:=\VB[\Sym^*,\sigma]=\left( \frac{\vv_n^*}{\norm{\vv_n^*}}\right)_{n=1}^\infty.
\]
There is a natural isomorphism between the dual space of $Q_\sigma(\Sym)\oplus\XX$ and $Q_\sigma(\Sym^*)\oplus\XX^*$. In turn, since $\BB$ is a basic sequence of $\XX^*$ and $\Sym^*$ is a locally convex subsymmetric sequence space, $Q_\sigma(\Sym_0^*)\oplus[\BB]$ can be identified with $\YY[\BB,\Sym^*,\sigma]$. Via this identification we obtain a dual pairing between $\YY[\BB,\Sym_0^*,\sigma]$ and $\YY[\XB,\Sym,\sigma]$ given by
\begin{align*}
(g,f) & \mapsto \langle Q_\sigma(g) ,Q_\sigma(f) \rangle+ L[\UB,\BB](P_\sigma(g)) ( L[\VB,\XB](P_\sigma(f)) )\\
& \langle Q_\sigma(g) ,Q_\sigma(f) \rangle+ L[\VB^*,\XB^*](P_\sigma(g)) ( L[\VB,\XB](P_\sigma(f)) )\\
&= \langle Q_\sigma(g) ,Q_\sigma(f) \rangle+ \langle P_\sigma(g) ,P_\sigma(f) \rangle\\
&= \langle f,g \rangle.
\end{align*}

$\ref{prop:basisY:2}$ The image of $\ee_k$ by the isomorphism provided by Lemma~\ref{lem:isomorphism} is, if $k\in\sigma_n$,
\[
\zz_k:=\left( \left(1-\frac{1}{\abs{\sigma_n}}\right) \ee_k -\frac{1}{\abs{\sigma_n}} \Ind_{\sigma_n\setminus\{ k\}}, \frac{\Lambda_{\aabs{\sigma_n}} }{\abs{\sigma_n}} \xx_n\right).
\]
Therefore,
\begin{align*}
\norm{ \ee_k}=\norm{ \zz_k}
&\approx \max\left\{ 1-\frac{1}{\abs{\sigma_n}}+\frac{ \Lambda_{\aabs{\sigma_n}-1}}{\abs{\sigma_n}},
\frac{\Lambda_{\aabs{\sigma_n}} }{\abs{\sigma_n}} \norm{ \xx_n}\right\}\\
&\approx \max\left\{ 1, \frac{\Lambda_{\aabs{\sigma_n}} }{\abs{\sigma_n}} \norm{ \xx_n}\right\}.
\end{align*}

$\ref{prop:basisY:3}$ follows from combining $\ref{prop:basisY:2}$ and $\ref{prop:basisY:5}$. In turn, $\ref{prop:basisY:4}$ is a consequence of $\ref{prop:basisY:2}$ and \ref{prop:basisY:3}.  The statements $\ref{prop:basisY:8}$ and $\ref{prop:basisY:6}$ follow from the identity
\begin{equation}\label{eq:projection}
S_{\cup_{k\in A}\sigma_k}[\EB,\YY](f)=\left(S_{\cup_{k\in A}\sigma_k}[\EB,\Sym] (Q_\sigma(f)) , S_A[\XB,\XX] \left(\sum_{n=1}^\infty \vv_n^*(f) \, \xx_n\right)\right).
\end{equation}

The proof of  $\ref{prop:basisY:10}$ goes along the lines of the corresponding statement  from \cite{AADK2018}, which also works in this more general setting. 

 To prove $\ref{prop:basisY:9}$,  for each $n\in\NN$  we choose a partition $(A_n,B_n)$ of $\sigma_n$ such that $0\le \abs{A_n} -  \abs{B_n} \le 1$. We have
$$
\norm{\Ind_{A_n}+\Ind_{B_n}}=\Lambda_{\aabs{\sigma_n}}{\|\xx_n\|}
$$
\mbox{ and }
\[\norm{\Ind_{A_n}-\Ind_{B_n}}=  \frac{\abs{\sigma_n}-\gamma_n}{\abs{\sigma_n}} \Lambda_{\aabs{\sigma_n}} +\frac{ \Lambda_{\aabs{\sigma_n}}}{{\aabs{\sigma_n}}}\norm{\xx_n} \gamma_n,
\]

where $\gamma_n=0$ if $\sigma_n$ is even and $\gamma_n=1$ otherwise.  Suppose that the unit vector basis of $\YY$ is UCC. Then, combining these estimates with 
\[
\norm{\Ind_{A_n}+\Ind_{B_n}}\approx \norm{\Ind_{A_n}-\Ind_{B_n}}, \quad n\in\NN
\] 
gives that $\XB$ is semi-normalized.

$\ref{prop:basisY:12}$ is clear.  Finally, by  Lemma~\ref{lem:TQGvsMB}, and taking into account that, by Lemma~\ref{lem:ULRP} and Lemma~\ref{lem:LorentzEq}, $d_1(\ww)=d_1(\uu)$ up to an equivalent norm, $\ref{prop:basisY:11}$ follows as  an easy consequence of $\ref{prop:basisY:9}$, $\ref{prop:basisY:12}$, and $\ref{prop:basisY:10}$.
\end{proof}

\section{ Existence of non $M$-bounded bases in Banach spaces}\label{sect:NMB}\noindent

\noindent This section is geared towards the construction of bases to which we will apply the DKK method with a purpose that will become clear below. We start with a bidimensional construction. 

Given $R\ge \sqrt{2}$, we consider the pair of vectors of $\FF^2$ given by 
\[
\hh_{1,R}=(1,0), \quad \hh_{2,R}=\left( 1-\frac{2}{R^2}, \frac{2}{R} \sqrt{1-\frac{1}{R^2}} \right).
\]
Notice that if $\alpha\in(0,\pi/4]$ is defined by $\sin(\alpha)=1/R$, then  $\hh_{2,R}=(\cos(2\alpha), \sin(2\alpha))$. We consider $\FF^2$ equipped with the Euclidean distance. 
\begin{lemma}\label{lem:R2Basis}
Given $R\ge \sqrt{2}$, the vectors
$\{\hh_{1,R},\hh_{2,R}\}$ form  a basis of $\FF^2$ whose biorthogonal functionals are
\[
\hh_{1,R}^*=\frac{R^2}{2 \sqrt{R^2-1} }(\sin(2\alpha), -\cos(2\alpha)), \quad \hh_{2,R}^*=\frac{R^2}{2 \sqrt{R^2-1}}(0,1),
\]
where $\alpha\in(0,\pi/4]$ is given by $\sin(\alpha)=1/R$.
Moreover,  
\[\norm{ \hh_{1,R}}=\norm{ \hh_{2,R}}=1, \norm{ \hh_{1,R}^*}=\norm{ \hh_{2,R}^*}\approx R, 
\norm{ \hh_{1,R}-\hh_{2,R}}=2/R,\] and 
\[
\sqrt{x^2+y^2} \le \norm{ x\, \hh_{1,R} +y\, \hh_{2,R}} \le x+y, \quad x,y\ge 0.
\]
\end{lemma}
\begin{proof}
It is a routine computation.
\end{proof}

Let  $\XB=(\xx_n)_{n=1}^\infty$ be a sequence in a quasi-Banach space $\XX$. For each $n\in\NN$ we consider the linear map
\begin{equation}\label{eq:R2toX}
L_n\colon\FF^2\to \XX, \quad L_n(1,0)=\xx_{2n-1},\; L_n(0,1)=\xx_{2n}.
\end{equation}
Given a sequence $\eta=(\lambda_n,\mu_n)_{n=1}^\infty$  in $\RR_+^2$ with $\lambda_n\mu_n>1$ for all $n\in\NN$,
we define a sequence $\XB_{\,\eta}=(\yy_n)_{n=1}^\infty$ in $\XX$ by
\[
\yy_{2n-1}=\lambda_n L_n (\hh_{1,\lambda_n\mu_n}), \quad \yy_{2n}=\lambda_n L_n (\hh_{2,\lambda_n\mu_n}), \quad n\in\NN.
\]

\begin{lemma}\label{lem:Ln}
Let $\XB=(\xx_n)_{n=1}^\infty$ be a semi-normalized $M$-bounded  basis in a quasi-Banach space $\XX$. Then, the operators   $(L_n)_{n=1}^\infty$ defined as in \eqref{eq:R2toX} are uniform isomorphisms.
\end{lemma}

\begin{proof}
Let $\kappa$ be the modulus of concavity of $\XX$. Let $c=\sup_n \norm{ \xx_n}$ and $d=\sup_n \norm{\xx_n^*}$. We have
\[
\frac{1}{d} \max\{\abs{x},\abs{y}\} \le \norm{ L_n(x,y)} \le \kappa c (\abs{x}+\abs{y}) \quad x,\,y\in\FF.\qedhere.
\]
\end{proof}

\begin{lemma}\label{lem:GB}
Let $\eta=(\lambda_n,\mu_n)_{n=1}^\infty$ be a sequence in  $\RR_+^2$ with $\lambda_n\mu_n>1$ for all $n\in\NN$. Let $\XB=(\xx_n)_{n=1}^\infty$ be an $M$-bounded semi-normalized basis in a quasi-Banach space $\XX$. Then $\XB_{\,\eta}=(\yy_n)_{n=1}^\infty$ is a basis of $\XX$ such that, if $\XB_{\,\eta}^*=(\yy_n^*)_{n=1}^\infty$ denotes its dual basis,
\begin{align*}
\norm{ \yy_{2n-1}^*}
&\approx \norm{ \yy_{2n}^*}\approx \mu_n,\\
 \norm{\yy_{2n}- \yy_{2n-1}} &\approx \frac{1}{\mu_n},\; \text{and}\\
\norm{ a\yy_{2n-1}+ b \yy_{2n}}&\approx\lambda_n(a+b)
\end{align*}
for $n\in\NN$ and $a$, $b\ge 0$. Moreover,
\begin{enumerate}[label={(\roman*)}, leftmargin=*, widest=ii]
\item\label{lem:GB:1} if $\XB$ is total, so is $\XB_{\,\eta}$, and
\item\label{lem:GB:2} if $\XB$ is equivalent to another basis $\XB'$, then $\XB_{\,\eta}$ is equivalent to $\XB'_{\,\eta}$.
\end{enumerate}
\end{lemma}

\begin{proof}
Let $(L_n)_{n=1}^\infty$ and  $(L_n^*)_{n=1}^\infty$ be as in \eqref{eq:R2toX} with respect to $\XB$ and $\XB^*$, respectively. Since $L_n^*(g)(L_n(f))=\langle g, f\rangle$ for all $f$, $g\in\FF^2$ and $n\in\NN$, $\XB_{\,\eta}$ is a basis of $\XX$ whose biorthogonal functionals $(\yy_n^*)_{n=1}^\infty$ are given by
\[
\yy_{2n-1}^*=\frac{1}{\lambda_n} L_n (\hh_{1,\lambda_n\mu_n}^*), \quad \yy_{2n}^*=\frac{1}{\lambda_n} L_n (\hh_{2,\lambda_n\mu_n}^*), \quad n\in\NN.
\]
Combining Lemmas~\ref{lem:R2Basis}  and \ref{lem:Ln} yields the desired estimates for $\XB_{\,\eta}$ and its dual basis. Since the vectors of $\XB^*$ are linear combinations of $(\yy_n^*)_{n=1}^\infty$, \ref{lem:GB:1} holds. In turn, since the vectors is $\XB_{\,\eta}$ are linear combinations of $\XB$ with coefficients that do not depend of the given basis $\XB$, \ref{lem:GB:2} holds.
\end{proof}

\begin{lemma}\label{lem:EqPositive}
Let $\UB=(\uu_n)_{n=1}^\infty$ be a semi-normalized unconditional basis of a quasi-Banach space $\UU$. Let $(\mu_n)_{n=1}^\infty$ be a sequence in $(1,\infty)$.  Set $\eta=(1,\mu_n)_{n=1}^\infty$. Then $(\UB^2)_{\eta}$ is equivalent to $\UB^2$ for non-negative scalars.
\end{lemma}

\begin{proof}
Denote $\UB^2=(\xx_n)_{n=1}^\infty$ and $(\UB^2)_{\eta}=(\yy_n)_{n=1}^\infty$. Notice that 
\[
R_n(f):=\xx_{2n-1}^*(f)\xx_{2n-1}+ \xx_{2n}^*(f)\xx_{2n}=\yy_{2n-1}^*(f)\yy_{2n-1}+ \yy_{2n}^*(f)\yy_{2n}
\]
for all $n\in\NN$ and $f\in \UU^2$. For $f=\sum_{n=1}^\infty b_n\, \xx_n\in\UU^2$ we have
\begin{align*}
\norm{ f }
&=\max\left\{ \norm{ \sum_{n=1}^\infty b_{2n-1}\, \uu_n},\norm{ \sum_{n=1}^\infty b_{2n}\, \uu_n}\right\}\\
&\approx \norm{ \sum_{n=1}^\infty (\abs{b_{2n-1}}+\abs{b_{2n}})\, \uu_n}\\
&\approx \norm{ \sum_{n=1}^\infty  \norm{ b_{2n-1} \, \xx_{2n-1}+ b_{2n} \, \xx_{2n} }  \uu_n}\\
&= \norm{ \sum_{n=1}^\infty  \norm{R_n(f)} \, \uu_n }.
\end{align*}

Let  $(L_n)_{n=1}^\infty$ be as in \eqref{eq:R2toX} with respect to be basis $\UB^2$. Pick an eventually null sequence of scalars $(a_n)_{n=1}^\infty$, and set  $f=\sum_{n=1}^\infty a_n\, \xx_n$ and $g=\sum_{n=1}^\infty a_n\, \yy_n$. Taking into account  Lemma~\ref{lem:Ln} we obtain
\begin{align*}
\norm{ g }
&\approx\norm{ \sum_{n=1}^\infty  \norm{ R_n(g) }  \uu_n}\\
&= \norm{\sum_{n=1}^\infty  \norm{  \yy_{2n-1}^*(g)\yy_{2n-1}+ \yy_{2n}^*(g)\yy_{2n}  } \uu_n}\\
&= \norm{ \sum_{n=1}^\infty  \norm{  L_n(a_{2n-1} \hh_{1,\mu_n} + a_{2n} \hh_{1,\mu_n}) }  \uu_n}\\
&\approx \norm{\sum_{n=1}^\infty  \norm{ a_{2n-1} \hh_{1,\mu_n} + a_{2n} \hh_{2,\mu_n} }  \uu_n}.\\
\end{align*}
Therefore, if $a_n\ge 0$ for all $n\in\NN$, applying Lemma~\ref{lem:R2Basis}  yields 
\[\norm{ g} \approx  \norm{ \sum_{n=1}^\infty  (a_{2n-1} + a_{2n})   \uu_n}
\approx \norm{ f}.\qedhere
\]
\end{proof}

\section{Squeeze-symmetric bases with large unconditonality parameters}\label{sect:LUC}
\noindent
In this section, we apply the DKK method, as developed in Section~\ref{sect:DKK}, to the bases we have constructed in Section~\ref{sect:NMB}.

\begin{lemma}\label{lem:DKKW}
Let $\Sym$ be a subsymmetric sequence space, let $\sigma$ be an ordered partition of $\NN$ with $s_n:=\abs{\sigma_{2n-1}}=\abs{\sigma_{2n}}\ge 2$ for all $n\in\NN$, let $\XB$ be a semi-normalized $M$-bounded basis of a quasi-Banach space $\XX$.
Set 
\[
\eta=\left( \frac{s_n}{\Lambda_{s_n}} ,\Lambda_{s_n}\right)_{n=1}^\infty,
\]
\[
R_n(a,b):=\norm{ a \Ind_{\sigma_{2n-1}} + b \Ind_{\sigma_{2n}}}_{\XB_\eta,\Sym,\sigma}, \quad a,b\in\FF,\;n\in\NN,
\]
and $M_n=\sum_{k=1}^{2n} \abs{\sigma_n}$ for all $n\in\NN$. Then, 
\begin{enumerate}[label=(\roman*), leftmargin=*,widest=iii]
\item\label{DKKW:1} the unit vector system is a  semi-normalized $M$-bounded basis of 
$
\YY:=\YY[\XB_{\,\eta},\Sym,\sigma]
$,
\item\label{DKKW:2} $R_n(a,b)\approx(a+b)s_n$ for $a$, $b\ge 0$  and $n\in\NN$,
\item\label{DKKW:3}  $R_n(-1,1)\approx 1$ for $n\in\NN$,
\item\label{DKKW:4} $\kl_{M_{n}}[\EB,\YY]\gtrsim s_n$ for $n\in\NN$.
\end{enumerate} 
\end{lemma}
\begin{proof}
Denote $\XB_{\,\eta}=(\yy_n)_{n=1}^\infty$. For all $a$, $b\in\FF$ we have
\[
R_n(a,b)=\norm{ a \Lambda_{\aabs{\sigma_{2n-1}}} \yy_{2n-1} +b \Lambda_{\aabs{\sigma_{2n}}} \yy_{2n}}
=\Lambda_{s_n} \norm{  a \yy_{2n-1} + b \yy_{2n}}.
\]
In light of this identity,  we obtain $\ref{DKKW:1}$,  $\ref{DKKW:2}$ and  $\ref{DKKW:3}$ by combining Lemma~\ref{lem:GB} with Proposition~\ref{prop:basisY}. Combining  $\ref{DKKW:2}$ and  $\ref{DKKW:3}$ gives $\ref{DKKW:4}$.
\end{proof}

\begin{theorem}\label{thm:A}
Let $\XX$ be a quasi-Banach space with an $M$-bounded basis $\XB$. Suppose that a locally convex symmetric sequence space $\Sym$ is complemented in $\XX$. Then $\XX$ has a $M$-bounded basis $\BB$ with 
\[
\kl_m[\BB,\XX]\gtrsim m, \quad m\in\NN.
\]
Moreover, if $\XB$ is total, the basis $\YB$ is total.
\end{theorem}

\begin{proof}
Consider the ordered partition of $\NN$ given by $\abs{\sigma_{2n-1}}=\abs{\sigma_{2n}}=2^{n}$ for all $n\in\NN$. The space
$P_\sigma(\Sym)\oplus\XX$ has a $M$-bounded semi-normalized basis. Applying Lemma~\ref{lem:DKKW}  to this basis yields a basis $\YB$ of
\[
\YY:=\YY[\XB_{\,\eta}, \Sym,\sigma] \approx Q_\sigma(\Sym)\oplus P_\sigma(\Sym)\oplus\XX\simeq \Sym\oplus\XX\simeq\XX
\]
with
\[
\kl_{2^{n+2}-1}[\YB,\YY]\gtrsim 2^n, \quad n\in\NN.
\]
From here, the desired estimate follows in a routine way.
\end{proof}

To contextualize the next result we point out that any Schauder basis $\XB$ of any superreflexive Banach space $\XX$ satisfies the estimate
\[
\kk_m[\XB,\XX] \lesssim m^{a}, \quad m\in\NN,
\]
for some $a<1$ (see \cite{AAW2019}).
\begin{corollary}
Let $\XX$ be a separable Banach space. Suppose that a symmetric sequence space $\Sym$ is complemented in $\XX$. Then $\XX$ has an $M$-bounded total basis $\BB$ with 
\[
\kl_m[\BB,\XX]\approx m, \quad m\in\NN.
\]
\end{corollary}

\begin{proof}
Combining  Theorem~\ref{thm:ExistsMB} with Theorem~\ref{thm:A} yields an $M$-bounded basis $\BB$ with  $\kl_m[\BB]\gtrsim m$ for $m\in\NN$. In light of \eqref{eq:BanachEstimate}, we are done.
\end{proof}

We next see applications to greedy-like bases.
\begin{theorem}\label{eq:MainA}
Let $\XX$ be a separable Banach space that has a symmetric sequence space $\Sym$ as a complemented subspace. Suppose also that  $(\Lambda_m[\Sym])_{m=1}^\infty$ has the LRP. Then $\XX$ has a squeeze-symmetric basis $\BB$ (bidemocratic if  $(\Lambda_m[\Sym])_{m=1}^\infty$  has additionally the URP)  with 
\[\kl_m[\BB,\XX] \approx \log (1+m)\quad \text{and}\quad \udf[\BB,\XX](m) \approx \Lambda_m[\Sym],\quad m\in\NN.\] Moreover, we can choose $\BB$ to be non quasi-greedy. 
\end{theorem}

\begin{proof}
Let $\sigma'=(\sigma_n')_{n=1}^\infty$ be the partition of $\NN$ given by $\abs{\sigma'_{2n-1}}=\abs{\sigma'_{2n}}=2^n$ for all $n\in\NN$, and let $(\sigma_n)_{n=1}^\infty$ be the partition of $\NN$ given by  $\abs{\sigma_n}=2^n$ for all $n\in\NN$.  
By Teorema~\ref{thm:ExistsMB}, the Banach space 
\[
\XX_0:=P_\sigma(\Sym) \oplus P_{\sigma'}(\Sym)\oplus \XX
\]
has a semi-normalized $M$-bounded basis,  say $\XB_0$.
A combination of Lemma~\ref{lem:DKKW} and Lemma~\ref{lem:isomorphism} gives that 
$\XX_1:=Q_{\sigma'}(\Sym)\oplus \XX_0$ has a semi-normalized $M$-bounded basis $\XB_1$ such that
$R_n(a,b)\approx  2^n(a+b)$ for $a$, $b\ge 0$ and $R_n(1,-1) \approx 1$, where
\[
R_n(a,b) = \norm{ a\Ind_{\sigma'_{2n-1} }[\XB_1,\XX_1] + b\Ind_{\sigma'_{2n} }[\XB_1,\XX_1] }.
\]
As noted in the proof of Theorem~\ref{thm:A}, this implies $\kl_m[\XB_1] \gtrsim m$.  By Theorem~\ref{prop:basisY}, the unit vector system  is a sequeeze-symmetric basis of of $\YY_1=\YY[\XB_1,\Sym,\sigma]$ whose dual basis is equivalent to the unit vector system of $\YY_2=\YY[\XB_2,\Sym_0^*,\sigma]$ for a suitable semi-normalized $M$-bounded basis $\XB_2$. 

Suppose that $(\Lambda_m[\Sym])_{m=1}^\infty$ has the URP. Then  $(\Lambda_m[\Sym_0^*])_{m=1}^\infty$ has the LRP and, hence,  the unit vector system of $\YY_2$  is squeeze-symmetric as well. In particular,
\[
\udf[\EB,\YY_1](m) \udf[\EB,\YY_2](m)  \approx \Lambda_m[\Sym] \Lambda_m[\Sym_0^*]\approx m, \quad m\in\NN.
\]
Hence, the unit vector system of $\YY_1$ is bidemocratic. By Lemma~\ref{lem:isomorphism}, 
\[
\YY_1\approx Q_\sigma(\Sym) \oplus \XX_1\approx \Sym\oplus \Sym\oplus \XX\approx \XX.
\]

Theorem~\ref{prop:basisY} also gives
\[
\kl_{2^n-1}[\EB,\YY_1]\gtrsim n, \quad n\in\NN.
\]
Therefore, $\kl_m[\EB,\YY_1]\gtrsim \log m$ for $m\ge 2$. 
To prove that $\EB$ is not a quasi-greedy basis of $\YY$, we consider the vectors
\[
f=\sum_{k=1+M_{2n-2}}^{M_{2n-1}} \frac{1}{\Lambda_{\aabs{\sigma_k}}} \Ind_{\sigma_k}, \quad
g=\sum_{k=1+M_{2n-1}}^{M_{2n}} \frac{1}{\Lambda_{\aabs{\sigma_k}}} \Ind_{\sigma_k}
\]
where $M_n=\sum_{k=1}^{n} \abs{\sigma’_{k}}$. We have $\norm{-f+g}=R_n(-1,1)\approx 1$ and  $\norm{g}=R_n(0,1)\approx 2^n$. Since $g$ is a greedy projection of $-f+g$, we are done.
\end{proof}

We are now in a position to prove that separable Hilbert spaces have  squeeze-symmetric bases with large unconditionality constants. In fact,  Theorem~\ref{eq:MainA} allows us to state a mild condition on a Banach space which ensures the existence of a  truncation quasi-greedy basis for which the estimate \eqref{eq:GE} is optimal.

\begin{corollary}\label{cor:SSLarge}
Let $\XX$ be a separable Banach space that has a subsymmetric sequence space $\Sym$ as a complemented subspace. Suppose also that $\XX$ has some nontrivial type (resp., cotype). Then, $\XX$ has a bidemocratic (resp., sequeeze-symmetric) total basis $\BB$ with 
\[
\kl_m[\BB,\XX] \approx \log (1+m),\quad m\in\NN.
\]
 Moreover, we can choose $\BB$ so that is not quasi-greedy.
\end{corollary}

\begin{proof}
Just combine  Proposition~\ref{prop:RadLURP} with Theorem~\ref{eq:MainA}.
\end{proof}

Note that  Corollary~\ref{cor:SSLarge} yields, in particular, the existence of bidemocratic non-quasi-greedy bases for a wide class of Banach spaces. In this regard, it improves the results in this direction obtained in \cite{AABBL2021}.

\begin{remark}
Let $\ww$  be a weight whose primitive sequence is doubling.  Let $1\le q\le r\le \infty$ be such that either $q>1$ or $r<\infty$. Suppose that a basis $\XB$ of a quasi-Banach space $\XX$ is squeezed between  $d_q(\ww)$ and $d_r(\ww)$. Then, by \cite{Ansorena2021}*{Lemma 2.5 and Equation (2.7)} and \cite{AAB2021}*{Remark 5.2}, 
\[
\kk_m[\XB,\XX] \lesssim (1+\log m)^{1/q-1/r}.
\]
We infer that the bidemocratic basis $\BB$ achieved in Proposition~\ref{cor:SSLarge} can not be such a way squeezed. 
So, if we choose $\XX$ to be a Hilbert space, $\BB$ is a counterexample that solves in the negative \cite{ABW2021}*{Question 5.3}.
\end{remark}

We close this section by using the  machinery we have developed to generalize the construction of a democratic non-UCC basis of $\ell_2$  carried out in \cite{AABW2021}* {Example 11.21}.

\begin{proposition}
Let $\Sym$ be a locally convex subsymmetric sequence space. Then $\Sym$ has a democratic Schauder basis $\XB$ with
\[
\udf[\XB,\XX] (m)\approx \ldf[\XB,\XX](m)\approx \Lambda_m[\Sym],\quad m\in \NN,
\]
and $\sldf[\XB,\XX] (m)\approx 1$ for $m\in\NN$.
\end{proposition}

\begin{proof}
Let $\sigma$ the an ordered partition of $\Sym$ with $\abs{\sigma_{2n-1}}=\abs{\sigma_{2n}}=n$. Set $\eta=(1,\Lambda_{s_n})_{n=1}^\infty$. Consider the unit vector system $\EB$ of the space
\[
\YY:=\YY[ (\VB[\Sym,\sigma])_{\,\eta},\Sym,\sigma]\approx Q_\sigma(\Sym)\oplus P_\sigma(\Sym)\approx \Sym.
\]
By Lemma~\ref{lem:DKKW}, $\sldf[\EB,\YY](2n)\approx 1$ for $n\in\NN$. Let $\tau$ be the ordered partition of $\NN$ given by $\abs{\tau_n}=n$ for all $n\in\NN$. By subsymmetry, $(\VB[\Sym,\tau])^2$ is equivalent to $\VB[\Sym,\sigma]$. Then, by Lemma~\ref{lem:GB}\ref{lem:GB:2} and Lemma~\ref{lem:EqPositive}, $\VB[\Sym,\sigma]$ and $(\VB[\Sym,\sigma])_{\,\eta}$ are equivalent for positive scalars. Consequently, the unit vector systems of $\YY$ and $\YY_1:=\YY[ (\VB[\Sym,\sigma], \Sym,\sigma])$ are equivalent for positive scalars. In turn, the unit vector system of $\YY_1$ is equivalent to that of $\Sym$. Applying these equivalences to constant coefficients sequences yields the desired result.
\end{proof}

\section{Measuring conditionality via the near unconditionality functions}\noindent
In Section~\ref{sect:LUC}, we have used the DKK method to obtain greedy-like bases with large unconditionality constants. In general, the conditionality of such bases in terms of their unconditionality constants has been extensively studied in the literature (see, e.g., \cite{AAB2021}, \cite{AAW2019}, \cite{AADK2018}, \cite{Ansorena2021},  \cite{BBG2017}, \cite{BBGHO2018}). In this section, we look at their conditionality using a different meausuring tool. First, we need some additional notation and terminology. Given a 
basis  $\XB=(\xx_n)_{n=1}^\infty$ of a 
quasi-Banach space $\XX$
we define 
\begin{align*}
\Cu=\Cu[\XB,\XX]&=\{f\in \XX \colon \norm{\Fou(f)}_\infty=\sup_n\abs{\xx_n^*(f)}\le 1\}, \mbox{ and}\\
A(f,a)&=\{n\in\NN \colon |\xx_n^*(f)|\ge a\},\quad f\in \XX, \, a>0.\\
\end{align*}
We say that basis $\XB$ is \emph{nearly unconditional} if for each $a\in(0,1]$ there is a constant $C$ such that, for any $f\in \Cu$,
\begin{equation}\label{nearlyuncdef}
\Vert S_A(f)\Vert\le C\Vert f\Vert,
\end{equation}
whenever $A\subseteq A(f,a)$. The \emph{near unconditionality function of the basis} is the function $\phi\colon(0,1]\rightarrow [1,+\infty)$ that gives the optimal $C$ for which \eqref{nearlyuncdef} holds.

The notion of near unconditionality was introduced by Elton in \cite{Elton1978}, and imported to greedy approximation theory by Dilworth et al.\ in \cite{DKK2003}. It is known that truncation quasi-greedy bases are nearly unconditional (see \cite{AABBL2022}*{Corollary 3.5} and \cite{DKK2003}*{Proposition 4.5}), so their conditionality can be studied by means of the near unconditionality function $\phi$. In this context, it is natural to ask how this approach compares with the standard way of studying it via the unconditionality parameters $(\kk_m)_{m=1}^{\infty}$.  As we shall see, despite the facts that 
the cardinality of the sets in the projections involved in \eqref{nearlyuncdef} is unrestricted, the results obtained for the unconditionality constants of quasi-greedy and truncation quasi-greedy bases also hold for $\phi$, in a sense that will become clear below. We begin with a lemma that allows us to compare both measures of conditionality.
\begin{lemma}\label{lemma: kmphi} 
Let $\XB$ be a semi-normalized M-bounded nearly unconditional basis of a $p$-Banach space $\XX$, $0<p\le 1$. Let $F\colon(0,\infty)\rightarrow [1,\infty)$ be a non-deincreasing function such that 
\[
\kk_m\ge F(m), \quad m\in \NN.
\]
Then, if $\alpha_1=\sup_{n}\norm{\xx_n}$ and  $\alpha_2=\sup_{n}\norm{\xx_n^*}$,
\[
\phi(a)\ge \frac{1}{4^{1/p}\alpha_1\alpha_2} F\left(a^{-p}\right), \quad  0<a\le 1. 
\]
\end{lemma}
\begin{proof}
Pick $m\in\NN$ such that $m\le a^{-p}<m+1$.
Then, fix $f\in \XX$ with $\norm{f}=1$, and $A\subseteq \NN$ with $\abs{A}\le m$. Set $g=
f/ \norm{\Fou(f)}_{\infty}$, so that $g\in \Cu$. Let $(A_1,A_2)$ be the partition of $A$ defined by
$$
A_1:=\left\lbrace n\in A: |\xx_n^*(g)|\le m^{-{1}/{p}}\right\rbrace. 
$$
We have
\begin{align*}
\norm{S_{A}(g)}^p
&\le \norm{S_{A_1}(g)}^p+\norm{S_{A_2}(g)}^p\\
&\le  \sum_{n\in A_1}\abs{ \xx_n^*(g)}^p \norm{\xx_n}^p+\phi^{\,p}\left(m^{-{{1}/{p}}}\right)\|g\|^p\\
&\le \frac{\abs{A_1}\alpha_1^p}{m}+\phi^{\,p}\left(m^{-{{1}/{p}}}\right)\|g\|^p.
\end{align*}
Thus, since $\abs{A_1}\le m$, $\norm{g} \norm{\Fou(f)}_\infty=1$, and $\norm{\Fou(f)}_\infty\le \alpha_2 $,
\[
\norm{S_{A}(f)}^p \le \alpha_1^p\alpha_2^p+\phi^{\,p}\left(m^{-{{1}/{p}}}\right)
\le 2 \alpha_1^p\alpha_2^p\phi^{\,p}\left(m^{-{{1}/{p}}}\right).
\]
Taking the supremum over $f$ and $A$ we obtain
\begin{equation}\label{eq:1/m}
\kk_m \le 2^{1/p}\alpha_1\alpha_2 \phi\left(m^{-{{1}/{p}}}\right).
\end{equation}
Since $\phi$ is non-increasing, combining \eqref{eq:1/m} with the inequality $\kk_{m+1}^p\le \kk_{2m}^p\le 2\kk_m^p$ gives
\[
F(a^{-p})\le F(m+1) \le \kk_{m+1} \le 4^{1/p}\alpha_1\alpha_2 \phi\left(m^{-{{1}/{p}}}\right)
\le  4^{1/p}\alpha_1\alpha_2\phi(a).\qedhere
\]
\end{proof}
Lemma~\ref{lemma: kmphi} tells us that, in a natural sense, lower bounds for $\kk_m$ are stronger than lower bounds for $\phi$. Thus, it allows us to transfer to the language of near unconditionality functions any result that guarantees the existence of bases with large unconditionality constants. For instance, we infer the following results.

\begin{corollary}\label{cor:MainA}
Let $\XX$ be a separable Banach space. Suppose that a symmetric sequence space $\Sym$ is complemented in $\XX$. Suppose also that  $(\Lambda_m[\Sym])_{m=1}^\infty$ has the LRP. Then $\XX$ has a squeeze symmetric (bidemocratic if  $(\Lambda_m[\Sym])_{m=1}^\infty$ also has the URP)  total basis $\BB$ with 
\[
\phi(a)\gtrsim 1-\log a, \quad 0<a\le 1,
\]
and $\udf[\BB,\XX](m) \approx \Lambda_m[\Sym]$ for $m\in\NN$. Moreover, $\BB$ is not quasi-greedy. 
\end{corollary}

\begin{corollary}\label{corollaryreflexiveqg}There is a reflexive Banach space $\XX$ with an almost greedy Schauder basis $\XB$ such that
\[
\phi(a)\gtrsim 1-\log a, \quad 0<a\le 1.
\]
\end{corollary}

\begin{corollary}\label{cor:DKK}
Let $\XX$ be a Banach space with a Schauder basis. Suppose that $\ell_p$, $1<p<\infty$, is complemented in $\XX$. Then $\XX$ has, for any $0<\epsilon<1$, an almost greedy Schauder basis $\BB$ with 
\[
\phi(a)\gtrsim \left( 1-\log a \right)^{1-\epsilon}, \quad 0<a\le 1,
\]
and $\udf[\BB,\XX](m) \approx m^{1/p}$ for $m\in\NN$.
\end{corollary}

In fact, Corollary~\ref{cor:MainA}, Corollary~\ref{corollaryreflexiveqg}  and Corollary~\ref{cor:DKK}  follow from combining Lemma~\ref{lemma: kmphi}  with Theorem~\ref{eq:MainA}, \cite{AADK2018}*{Theorem 4.18} and  \cite{AADK2018}*{Theorem 4.5}, respectively. We refer the reader to \cite{AADK2018} for more examples of Banach spaces with large unconditionality constants, then large  near unconditionality functions.

Theorem~\ref{lemma: upperboundtqg} and Theorem~\ref{prop: superrefphi} below, which improve  \cite{AAW2021b}*{Theorem 5.1} and \cite{AAGHR2015}*{Theorem 1.1}, respectively,  prove that Corollaries~\ref{cor:MainA}, \ref{corollaryreflexiveqg}  and \ref{cor:DKK} are optimal in the sense that the  near unconditionality functions of the bases we obtain are such large as possible. 

\begin{theorem}\label{lemma: upperboundtqg}
Let $\XB$ be a truncation quasi-greedy basis of a $p$-Banach space $\XX$. Then,
\[
\phi(a)\lesssim \left( 1-\log a \right)^{1/p}, \quad 0<a\le 1.
\]
\end{theorem}

\begin{theorem}\label{prop: superrefphi}
Let $\XB$ be a quasi-greedy basis of a superreflexive Banach space $\XX$. Then, there is $0<\epsilon<1$ such that
\[
\phi(a)\lesssim \left(1- \log a \right)^{1-\epsilon}, \quad 0<a\le 1.
\]
\end{theorem}

Before facing the proof of these results, we bring up an auxiliary lemma.
\begin{lemma}[See \cite{AABBL2021}*{Lemma 2.3}, \cite{AABW2021}*{Proposition 4.16} or \cite{AAW2021b}*{Lemma 5.2}]\label{lem:truncation quasi-greedyQU}
Suppose that $\XB=(\xx_n)_{n=1}^\infty$ is a semi-normalized truncation quasi-greedy basis of a quasi-Banach space $\XX$. Then there is a constant $C$ such that
\begin{align*}
\left\Vert \sum_{n\in A} a_n\, \xx_n\right\Vert \le& C \Vert f\Vert
\end{align*}
for every $f\in \XX$, every finite set $A\subseteq\NN$, and every finite family $(a_n)_{n\in A}$ such that $\max_{n\in A}|a_n|\le \min_{n\in A} |\xx_n^{\ast}(f)|$.  
\end{lemma}

\begin{proof}[Proof of Theorems~\ref{lemma: upperboundtqg} and \ref{prop: superrefphi}]
In the superreflexive case,  an application of \cite{ABW2021}*{Theorem 2.3} gives $C_0\in[1,\infty)$ and $p\in(1,\infty)$ such that 
 \begin{equation}\label{eq:AAGHR}
\norm{ \sum_{k=1}^m f_k}^p\le C_0^p  \sum_{k=1}^m\norm{f_k}^p
\end{equation}
for all $m\in \NN$ and all disjointly supported families  $(f_k)_{k=1}^m$ in $\langle \XB\rangle$ such that
\[
\max_{n\in \supp(f_{k})}\abs{\xx_n^*(f_{k})} \le \min_{n\in \supp(f_{k-1})}\abs{\xx_n^*(f_{k-1})}, \quad 2\le k\le m. 
\]
In turn, in the locally $p$-convex case for all values of  $0<p\le 1$, \eqref{eq:AAGHR} holds with $C_0=1$. Fix $f\in \Cu$, $0<a\le 1$ and $A\subseteq A(f,a)$. Pick $n\in\NN$ such that $2^{-n}< a\le 2^{1-n}$. Consider the partititon $(A_k)_{k=0}^n$ of $A$ given by $A_0=A\cap A(f,1)$, 
\[
A_k=A\cap \left(A(f,2^{-k})\setminus A(f,2^{1-k})\right),\quad  1\le k\le n-1,
\]
and $A_n=A\setminus A(f,2^{1-n})$. Let $C$ be the constant provided by Lemma~\ref{lem:truncation quasi-greedyQU}. In both cases we have
\[
\norm{S_{A_k}(f)}=2 \norm{S_{A_k}(f/2)}\le  2C \norm{f}, \quad 0\le k \le n.
\]
Hence, 
\begin{align*}
\norm{S_A(f)}^p=&\norm{ \sum_{k=0}^{n} S_{A_k}(f)}^p\\
&\le C_0^p\sum_{k=0}^{n}\norm{ S_{A_k}(f)}^p\\
&\le  (1+n) 2^p C^p C_0^p \norm{f}^p \\
&\le (2-\log_2 a) 2^p C^p C_0^p  \norm{f}^p.
\end{align*}
Taking the supremum over $A$ and $f$ we obtain
\[
\phi(a) \le  2 C C_0 (2-\log_2 a)^{1/p}, \quad 0<a\le 1.\qedhere
\]
\end{proof} 




\end{document}